\documentclass[twoside,11pt,reqno]{amsart}
\usepackage{amsmath,amssymb,amscd,mathrsfs,amscd}
\usepackage{graphics,verbatim}

\let\oldsection=\section

\newcommand{\p}[1]{\ensuremath{\overline{#1}}}
\newcommand{\losemi}{{\otimes \kern -.78em \ltimes}}
\newcommand{\rosemi}{{\otimes \kern -.78em \rtimes}}

\newcommand{\Hom}{\ensuremath{\operatorname{Hom}}}
\newcommand{\End}{\ensuremath{\operatorname{End}}}

\newcommand{\Ker}{\ensuremath{\operatorname{Ker} }}

\newcommand{\Ext}{\operatorname{Ext}}

\newcommand{\0}{\bar 0}
\newcommand{\1}{\bar 1}
\newcommand{\Z}{\mathbb{Z}}
\newcommand{\C}{\mathbb{C}}

\newcommand{\setof}[2]{\ensuremath{\left\{ #1 \:|\: #2 \right\}}}
\newcommand{\gl}{\ensuremath{\mathfrak{gl}}}
\newcommand{\g}{\ensuremath{\mathfrak{g}}}

\newcommand{\rk}{\operatorname{rank}}
\newcommand{\X}{\mathcal{X}}

\newcommand{\bj}{\mathbf{j}}

\newcommand{\fg}{\ensuremath{\mathfrak{g}}}
\newcommand{\fb}{\ensuremath{\mathfrak{b}}}

\newcommand{\fh}{\ensuremath{\mathfrak{h}}}

\newcommand{\fq}{\ensuremath{\mathfrak{q}}}
\newcommand{\fp}{\ensuremath{\mathfrak{p}}}
\newcommand{\ft}{\ensuremath{\mathfrak{t}}}

\newcommand{\V}{\mathcal{V}}
\newcommand{\HH}{\operatorname{H}}
\newcommand{\F}{\mathcal{F}}

\newtheorem{Df}{Definition}[subsection]
\newtheorem{theorem}[Df]{Theorem}

\newtheorem{prop}[Df]{Proposition}
 
\newtheorem{cor}[Df]{Corollary}

\numberwithin{equation}{subsection}

\begin{document}
\title{Complexity and Module Varieties for classical Lie superalgebras}

\author{Brian D. Boe }
\address{Department of Mathematics \\
          University of Georgia \\
          Athens, GA 30602}
\email{brian@math.uga.edu}
\author{Jonathan R. Kujawa}
\address{Department of Mathematics \\
          University of Oklahoma \\
          Norman, OK 73019}
\thanks{Research of the second author was partially supported by NSF grant
DMS-0734226}\
\email{kujawa@math.ou.edu}
\author{Daniel K. Nakano}
\address{Department of Mathematics \\
          University of Georgia \\
          Athens, GA 30602}
\thanks{Research of the third author was partially supported by NSF
grant  DMS-0654169}\
\email{nakano@math.uga.edu}
\date{\today}
\subjclass[2000]{Primary 17B56, 17B10; Secondary 13A50}

\begin{abstract} Let ${\mathfrak g}={\mathfrak g}_{\0}\oplus {\mathfrak g}_{\1}$ be a classical Lie superalgebra 
and ${\mathcal F}$ be the category of finite dimensional ${\mathfrak g}$-supermodules which are semisimple over 
${\mathfrak g}_{\0}$. In this paper we investigate the homological properties of the category ${\mathcal F}$. 
In particular we prove that ${\mathcal F}$ is self-injective in the sense that all projective supermodules 
are injective. We also show that all supermodules in ${\mathcal F}$ admit a projective resolution with polynomial rate of growth and, hence, one can study complexity in $\mathcal{F}$. 
If ${\mathfrak g}$ is a Type~I Lie superalgebra we introduce support varieties which detect projectivity and are related to the associated varieties of Duflo and Serganova.  If in addition $\fg$ has a (strong) duality then we prove that the conditions 
of being tilting or projective are equivalent. 
\end{abstract}

\maketitle

\section{Introduction}\label{S:intro}

\subsection{} Let $A$ be a finite dimensional cocommutative Hopf algebra 
over an algebraically closed field $k$ (or equivalently a finite group scheme). It is well-known 
that $A$ is a Frobenius algebra meaning that an $A$-module is projective if and only if 
it is injective. Moreover, every finite dimensional $A$-module admits a minimal projective resolution 
whose terms have dimension which increase at a polynomial rate of growth. This rate of growth, 
called the {\em complexity} was first introduced by Alperin \cite{Al} in 1977. 
Carlson \cite{Car:83} later introduced the idea of the support variety of a module (for group algebras) whose dimension 
coincides with the complexity 
of the module. Support varieties were later defined for other finite dimensional cocommutative 
Hopf algebras \cite{FPa, NP, SFB1, SFB2}. Only recently has a general theory for arbitrary 
$A$ been developed by Friedlander and Pevtsova \cite{FPe} via $\pi$-points. 

Now let ${\mathfrak g}={\mathfrak g}_{\0}\oplus {\mathfrak g}_{\1}$ be a classical Lie superalgebra 
over the complex numbers ${\mathbb C}$. Here we are assuming that ${\mathfrak g}_{\0}$ is a reductive 
Lie algebra (but \emph{not} that $\fg$ is simple). An important category of ${\mathfrak g}$-supermodules is the category ${\mathcal F}$ 
of finite dimensional ${\mathfrak g}$-supermodules which are completely reducible over ${\mathfrak g}_{\0}$. 
The category ${\mathcal F}$ has enough projectives and is in general not semisimple. For Type I 
Lie superalgebras the category ${\mathcal F}$ is a highest weight category as defined by Cline, 
Parshall and Scott \cite{CPS1}. In \cite{BKN1,BKN2} the authors initiated a study of support varieties for 
${\mathcal F}$ using relative cohomology for the pair $({\mathfrak g},{\mathfrak g}_{\0})$. 
The construction of these support varieties provides a homological framework which encapsulates the 
important combinatorial notions of atypicality and defect as defined by Kac and Wakimoto. 
However, it should be noted that these support varieties do not have the property that their dimension 
is equal to the rate of growth of the minimal projective resolution for a supermodule. For example, for the simple 
Lie superalgebra ${\mathfrak g}=\mathfrak{gl}(1|1)$ the trivial supermodule has complexity equal to two, but the 
dimension of the $({\mathfrak g},{\mathfrak g}_{\0})$ support is one (which is also the atypicality).

\subsection{} Throughout \cite{BKN1,BKN2} it was observed that many of the features of the category ${\mathcal F}$ 
were similar to that of the category of modules, $\text{Mod}(A)$, for a finite dimensional cocommutative Hopf algebra 
$A$. The main theme of this paper is to make this statement more precise. For example, in Theorem \ref{T:projectiveres} we show that 
every supermodule in $\mathcal{F}$ admits a projective resolution with a polynomial rate of growth. This implies that 
one can define the notion of the complexity of a supermodule in the category ${\mathcal F}$. We also prove 
in Section \ref{S:complexity} that a supermodule is projective in ${\mathcal F}$ if and only if it is injective. Combining these 
results, one can then say that a supermodule has complexity zero in ${\mathcal F}$ if and only if it is projective. 

The main difference in the blocks for the category ${\mathcal F}$ as opposed to $\text{Mod}(A)$ is that 
the non-semisimple blocks usually contain infinitely many simple supermodules. For Type I Lie superalgebras this 
forces us to consider a hybrid of two different theories: modules over self-injective algebras and modules 
over quasi-hereditary algebras. We should remark that quasi-hereditary algebras (cf.\ \cite{CPS1}) with only 
finitely many simple modules have finite projective dimension.  If such an algebra is self-injective then it must 
be semisimple. So the intersection of these two theories is non-trivial only when there exist infinitely 
many simple objects. Self-injective quasi-hereditary algebras have also arisen as infinite dimensional 
quiver algebras associated with rhombohedral tilings in the work of Chuang and Turner \cite{CT}. 

In the Type I setting, by definition ${\mathfrak g}$ admits a compatible ${\mathbb Z}$-grading ${\mathfrak g}={\mathfrak g}_{-1}
\oplus {\mathfrak g}_{0} \oplus {\mathfrak g}_{1}$ where ${\mathfrak g}_{\0}={\mathfrak g}_{0}$ 
and ${\mathfrak g}_{\1}={\mathfrak g}_{-1}\oplus {\mathfrak g}_{1}$. Examples of Type I Lie superalgebras are $\gl (m|n)$ and the simple Lie superalgebras of type $A(m,n)$, $C(n)$, and $P(n)$ in the Kac classification \cite[Proposition 2.1.2]{Kac1}.  For Type I Lie superalgebras one 
can define the notion of tilting supermodules in $\mathcal{F}$. In Section \ref{S:typeI} we present new results which give 
an equivalent criterion for a supermodule to be tilting via the use of (cohomological) support varieties ${\mathcal V}_{
{\mathfrak g}_{\pm 1}}(M)$ over the Lie superalgebras ${\mathfrak g}_{\pm 1}$. For an arbitrary Lie superalgebra $\fg$ and $M\in {\mathcal F}$, Duflo and 
Serganova \cite{dufloserganova} constructed an ``associated variety'' for $M$ denoted by $\X_{M}$. 
In Section \ref{SS:detectingprojectives}, we demonstrate that when $\fg$ is Type I one has ${\mathcal V}_{{\mathfrak g}_{-1}}(M)
\cup {\mathcal V}_{{\mathfrak g}_{1}}(M)\subseteq \X_{M}$ and $\V_{\fg_{\pm 1}}(M)= \fg_{\pm 1} \cap \X_{M}$. As far as the authors know, this is the first time 
that the theories of cohomological support varieties and associated varieties have been connected\footnote{After completing this work, the authors learned in a conversation with Duflo and Serganova that they initially considered the varieties $\mathfrak{g}_{\pm 1} \cap \X_{M}$.  However, to our knowledge, they did not obtain a cohomological description nor pursue their study.}. Furthermore, 
we prove that a supermodule $M$ in ${\mathcal F}$ is tilting if and only if it is projective (see Corollary \ref{C:BKNproj3}). 
From our results, we demonstrate in Theorem \ref{T:DSproj} that one can recover the Duflo-Serganova projectivity criterion 
\cite[Theorem 3.4]{dufloserganova} in Type I via the sequence of equivalences
$$
M \text{ is projective } \iff M \text{ is tilting} \iff \V_{\g_{\pm 1}}(M)=0 \iff \X_{M}=0. 
$$ 
The first equivalence is known for $\gl (m|n)$ by work of Brundan \cite{brundan1,brundan2}, and for $C(n)$ by work of Cheng, Wang, and Zhang \cite{CWZ}.  There it is a consequence of entirely different techniques involving translation functors.  To the authors' knowledge the result is new for type $P(n)$.

The authors wish to thank Michel Duflo, William Graham, Henning Krause, and Vera Serganova for useful discussions.  A portion of this work was done while the second author visited the Mathematical Sciences Research Institute in Berkeley, CA.  He would like to thank the administration and staff of MSRI for their hospitality and especially the organizers of the ``Combinatorial Representation Theory'' and ``Representation Theory of Finite Groups and Related Topics'' programs for an exceptionally stimulating semester.

\section{Projective Resolutions and Complexity for ${\mathcal F}$} \label{S:complexity}

\subsection{Notation}\label{S:prelims} We will use the notation and conventions developed in 
\cite{BKN1, BKN2}. For more details we refer the reader to \cite[Section 2.1]{BKN2}. 

We will work over the complex numbers $\C$ throughout this paper. Let ${\mathfrak g}$ be a 
\emph{Lie superalgebra}; that is, a $\Z_{2}$-graded vector space 
$\g=\g_{\0}\oplus \g_{\1}$ with a bracket operation $[\;,\;]:\g \otimes \g \to \g$ which preserves the 
$\Z_{2}$-grading and satisfies graded versions of the usual Lie bracket axioms.  The subspace 
$\g_{\0}$ is a Lie algebra under the bracket and ${\mathfrak g}_{\1}$ is a $\g_{\0}$-module. A finite dimensional Lie superalgebra 
${\mathfrak g}$ is called \emph{classical} if there is a connected reductive algebraic group $G_{\0}$ such 
that $\operatorname{Lie}(G_{\0})=\g_{\0},$ and an action of $G_{\0}$ on $\g_{\1}$ which differentiates to 
the adjoint action of $\g_{\0}$ on $\g_{\1}.$  (Note that, unlike in Kac's original definition \cite{Kac1}, we do \emph{not} require a classical Lie superalgebra to be simple.) If $\g$ is classical, then $\g_{\1}$ is semisimple as a $\g_{\0}$-module.  
A \emph{basic classical} Lie superalgebra is a classical Lie superalgebra with a nondegenerate invariant supersymmetric 
even bilinear form (cf.\ \cite{Kac1}).  

Let $U(\g)$ be the universal enveloping superalgebra of ${\mathfrak g}$. The objects of the category of $\g$-supermodules 
are all $\Z_{2}$-graded left $U(\g)$-modules.  To describe the morphisms we first recall that if $M$ and $N$ are $\Z_{2}$-graded, then $\Hom_{\C}(M,N)$ is naturally $\Z_{2}$-graded by setting $\p{f}=r \in \Z_{2}$ if $f(M_{i}) \subseteq N_{i+r}$ for $i \in \Z_{2}$.  For $\fg$-supermodules $M$ and $N$ a homogeneous morphism $f : M \to N$ is a homogeneous linear map which satisfies 
\[
f(xm)=(-1)^{\p{f} \; \p{x}}xf(m)
\] for all homogeneous $x \in \fg$.  Here and elsewhere we write $\p{v} \in \Z_{2}$ for the degree of a homogenous element $v$ of a $\Z_{2}$-graded vector space.  Furthermore, we use the convention that we only state conditions for a homogenous element with the general case given by linearity.  Given ${\mathfrak g}$-supermodules $M$ and $N$ one can use the 
antipode and coproduct of $U({\mathfrak g})$ to define a ${\mathfrak g}$-supermodule 
structure on the contragradient dual $M^{*}$ and the tensor product $M\otimes N$. A supermodule is 
\emph{finitely semisimple} if it decomposes into a direct sum of finite dimensional simple supermodules.  

For the remainder of the paper the term ``module'' will be used to refer to supermodules.  For example, if $N$ is a submodule of $M$ then we mean that $N$ is a subsupermodule; that is, $N$ is stable under the action of $U(\fg )$ and $N_{r}=N\cap M_{r}$ for $r \in \Z_{2}$.

In general an arbitrary subcategory of $\g$-modules is not an abelian category.  However, the \emph{underlying even category}, 
consisting of the same objects but only the even (i.e.\ degree preserving) morphisms, is an abelian category. One can use this and the parity 
change functor $\Pi,$ which interchanges the $\Z_{2}$-grading of a module, to make use of 
the tools of homological algebra (cf.\ \cite[Section~3.5]{BKN2}).  Given a category, $\mathcal{C},$ of $\fg$-modules 
and objects $M, N$ in $\mathcal{C},$ we write 
\[
\Ext_{\mathcal{C}}^{d}(M,N)
\] for the degree $d$ extensions between $N$ and $M$ in the category $\mathcal{C}.$
In this paper we will be mainly interested in the case when ${\mathcal C}={\mathcal F}({\mathfrak g},{\mathfrak g}_{\0})=:
{\mathcal F}$ which is the full subcategory of finite dimensional $\fg$-modules which are finitely semisimple over ${\mathfrak g}_{\0}$. 

\subsection{Self-injectivity}\label{SS:selfinjective} Let ${\mathfrak g}$ be a classical Lie superalgebra and $S$ be a simple finite dimensional 
${\mathfrak g}_{\0}$-module. Then $U({\mathfrak g})\otimes_{U({\mathfrak g}_{\0})}S$ is a projective module 
in ${\mathcal F}$. Moreover, $\Hom_{U({\mathfrak g}_{\0})}(U({\mathfrak g}),S)$ is an injective 
module in ${\mathcal F}$ via the right regular action of $\fg$ on $U(\fg)$; namely, 
$(x.f)(y)=(-1)^{(\p{f} +\p{y}) \p{x}}f(yx)$ for all homogeneous $f \in \Hom_{U({\mathfrak g}_{\0})}(U({\mathfrak g}),S)$ and homogeneous $x,y \in U(\fg)$. Let $N=\dim_{\C} {\mathfrak g}_{\1}$ and $\delta = \Lambda^{N}({\mathfrak g}_{\1})$. Then 
$\delta$ is a one dimensional $U({\mathfrak g}_{\0})$-module.  The following proposition was originally proven by Gorelik in the case when $\delta$ is the trivial $\fg_{\0}$-module \cite[Theorem 3.2.3]{gorelik}. 

\begin{prop}\label{P:Hom-TensorIso} Let $S$ be a simple finite dimensional $U({\mathfrak g}_{\0})$-module. Then we have the following 
isomorphism of $U({\mathfrak g})$-modules: 
$$U({\mathfrak g})\otimes_{U({\mathfrak g}_{\0})}S \cong 
\operatorname{Hom}_{U({\mathfrak g}_{\0})}(U({\mathfrak g}),S\otimes \delta).$$ 
\end{prop} 

\begin{proof} Set $M_{1}=\operatorname{Hom}_{U({\mathfrak g}_{\0})}(U({\mathfrak g}),S\otimes \delta)$ 
and $M_{2}=U({\mathfrak g})\otimes_{U({\mathfrak g}_{\0})}S $. Let $\text{ev}:M_{1}
\rightarrow S\otimes \delta$ be the evaluation homomorphism defined by $\text{ev}(f)=f(1)$. As 
a $U({\mathfrak g}_{\0})$-module, 
$$M_{2}\cong \Lambda^{\bullet}({\mathfrak g}_{\1})
\otimes_{{\mathbb C}}S.$$
Let $\Phi:M_{2}\rightarrow S\otimes \delta$ be the $U({\mathfrak g}_{\0})$-homomorphism with 
$\Phi(\Lambda^{j}(\fg_{\1})\otimes S)=0$ for $j\neq N$ and $\Phi$ switching (in the graded sense) the factors on $\Lambda^{N}({\mathfrak g}_{\1})
\otimes S$. By the universal property for $M_{1}$ there exists a $U({\mathfrak g})$-homomorphism 
$\Psi:M_{2}\rightarrow M_{1}$ defined by $\Psi(x)(g)=(-1)^{\p{g} \; \p{x}}\Phi(g.x)$ such that $\text{ev}\circ \Psi=\Phi$. 
Since $\dim_{\C} M_{1}=\dim_{\C} M_{2}$, it suffices to show that $\Psi$ is injective. 

Suppose that $\Psi(x)=0$. Then $\Psi(x)(g)=0$ for all $g\in U(\fg)$, thus $\Phi(g.x)=0$ for 
all $g \in U(\fg )$. Let $\{y_{1},y_{2},\dots,y_{N}\}$ be a basis for ${\mathfrak g}_{\1}$ and identify 
$\delta=\C\cdot (y_{1}\wedge y_{2} \wedge \dots \wedge y_{N})$. If $x\neq 0$ we can choose 
$g \in U(\fg )$ such that $g.x=\delta \otimes s+ z$ where $z$ only 
involves tensors whose first component does not have a term in $\Lambda^{N}(\fg_{\1})$. This contradicts 
the fact that $\Phi(g.x)=0$, thus $x=0$ and $M_{1}\cong M_{2}$. 
\end{proof} 

We will first establish that the category ${\mathcal F}$ is self-injective; that is, 
a module $P$ in ${\mathcal F}$ is projective if and only if $P$ is injective. 

\begin{prop}\label{P:self-injective} Let $M\in {\mathcal F}$. Then $M$ is a projective module in ${\mathcal F}$ 
if and only if it is injective in ${\mathcal F}$. 
\end{prop} 

\begin{proof} Let $P(T)$ be the projective cover for the simple $U({\mathfrak g})$-module $T$ in $\F$. 
By Frobenius reciprocity, 
$$\Hom_{U({\mathfrak g})}(U({\mathfrak g})\otimes_{U({\mathfrak g}_{\0})}S,T)\cong 
\Hom_{U({\mathfrak g}_{\0})}(S,T).$$
Choose a simple $U({\mathfrak g}_{\0})$-module $S$ such that $S$ is a $U({\mathfrak g}_{\0})$ direct 
summand of $T$. Then 
from the isomorphism above $P(T)$ is a $U({\mathfrak g})$-summand of  
$U({\mathfrak g})\otimes_{U({\mathfrak g}_{\0})}S$. From Proposition~\ref{P:Hom-TensorIso}, 
$P(T)$ is a direct summand of the injective module $\Hom_{U({\mathfrak g}_{\0})}(U({\mathfrak g}), 
S\otimes \delta)$. 
Hence, $P(T)$ is injective in ${\mathcal F}$. A dual argument can be used to prove that an injective 
${\mathcal F}$-module is also projective. 
\end{proof}

\subsection{}\label{SS:superkoszul} Let $V=V_{\0}\oplus V_{\1}$ be a superspace (i.e., a ${\mathbb Z}_{2}$-graded 
vector space). Recall that for $x\in V_{i},\ i\in\Z_{2}$, we write $\p{x} = i$. We write $S_{\sup}(V)$ and $\Lambda_{\sup}(V)$ for the graded commutative and graded exterior algebras.  The following 
result provides a super analogue of the Koszul complex for the symmetric algebra.

\begin{prop}\label{P:superkoszul}  Let $V$ be a superspace.  Let $s > 0$ and consider the complex 
\[  
\mathcal{K}=\mathcal{K}(V)=\left( \dotsb \to  S^{s-p}_{\sup}(V) \otimes \Lambda^{p}_{\sup}(V) \xrightarrow{d_{p}}  
S^{s-p+1}_{\sup}(V) \otimes \Lambda^{p-1}_{\sup}(V) \to \dotsb\right), 
\]
with differential $d_{p}$ given by 
\[
a \otimes x_{1} \wedge \dotsb \wedge x_{p} \mapsto \sum_{i=1}^{p} (-1)^{i+1+\gamma_{i}}ax_{i} \otimes x_{1}\wedge \dotsb 
\wedge \widehat{x}_{i} \wedge \dotsb \wedge x_{p},
\] 
for homogeneous elements $x_{1},\dots, x_{p}\in V$, where $\gamma_{i}=  \p{x}_{i}(\p{x}_{1}+\dotsb +\p{x}_{i-1}).$  Then this complex is exact.

\end{prop}

\begin{proof}  A direct calculation verifies that $d^{2} = 0$.  To show that the complex is exact we proceed by induction on $k:=\dim_{\C} V_{\1}.$  For the base case of $k=0$ the result follows from the 
fact that then $\mathcal{K}(V)$ is the usual Koszul complex (cf.\ \cite[3.1.4(4), Appendix D.13]{kumar}).  Now assume that the above 
complex is exact for all $s >0$ and all superspaces $W$ for which $\dim_{\C} W_{\1}=k-1$ and consider the case when $\dim_{\C} 
V_{\1} = k >0.$  Let $l=\dim_{\C} V_{\0}$ and fix a homogeneous basis $e_{1}, \dotsc , e_{k}, e_{k+1}, \dotsc e_{k+l}$ for $V$
such that $\p{e}_{i}=\1$ for $i=1, \dotsc ,k$ and $\p{e}_{i}=\0$ for $i=k+1, \dotsc , k+l.$  Then 
\[
\mathcal{K}_{p} := S_{\sup}^{s-p}(V) \otimes \Lambda_{\sup}^{p}(V)
\]
has as a basis elements of the form 
\[
a \otimes e_{\bj}:=a \otimes e_{j_{1}} \wedge \dotsb \wedge e_{j_{p}},
\] where $a$ ranges over the monomial basis of $S^{s-p}_{\sup}(V)$ and $\bj = (j_{1}, \dotsc , j_{p})$ ranges over all 
ordered $p$-tuples of elements of $\{ 1, \dotsc , k+l \}$ subject to the constraint that $k+1, \dotsc , k+l$ each 
appear at most once.   We define the \emph{$e_{k}$-degree} of the basis element $a \otimes e_{\bj}$ to be the 
number of times the basis element $e_{k}$ appears in $e_{j_{1}} \wedge \dotsb \wedge e_{j_{p}}.$  Note that 
since $e_{k}$ is odd (ie.\ degree $\1$), the $e_{k}$-degree of $a \otimes e_{\bj}$ can have value $0,1,\dotsc , p.$

Define a filtration on $\mathcal{K}$ 
\[
0 =\mathcal{K}(-1) \subseteq \mathcal{K}(0) \subseteq \mathcal{K}(1) \subseteq \dotsb 
\] by setting $\mathcal{K}(n)_{p}$ (for $n=-1,0,1,\dotsc $) equal to the span of all basis elements in 
$\mathcal{K}_{p}$ of $e_{k}$-degree less than or equal to $n.$  Observe that $d_{p}$ restricts to give a 
differential $d_{p}: \mathcal{K}(n)_{p} \to \mathcal{K}(n)_{p-1}.$    Since this filtration is convergent 
and bounded below, there is a spectral sequence \cite[Theorem E.4]{kumar}:
\begin{equation}\label{E:spectral}
E_{n,t}^{1} = \HH_{n+t}\left( \mathcal{K}(n)/\mathcal{K}(n-1)\right) \Rightarrow \HH_{n+t}(\mathcal{K}).
\end{equation}
Thus it suffices to show that the quotient complex $\mathcal{K}(n)/\mathcal{K}(n-1)$ is exact.

To prove this, we induct on $n.$  The base case is $n=0:$  $\mathcal{K}(0)/\mathcal{K}(-1)= \mathcal{K}(0)$.  
Let $\widetilde{V}$ denote the span of $e_{1}, \dotsc , e_{k-1}, e_{k+1}, \dotsc , e_{k+l}.$  Then
\[
\mathcal{K}(0)_{p} = S_{\sup}^{s-p}(V) \otimes \Lambda_{\sup}^{p}(\widetilde{V}),
\]
and, since $e_{k}$ is odd, we have the decomposition as vector spaces 
\[
S_{\sup}^{s-p}(V) = S_{\sup}^{s-p}(\widetilde{V}) \oplus e_{k}S_{\sup}^{s-p-1}(\widetilde{V}).
\] Tensoring this decomposition by $\Lambda_{\sup}^{p}(\widetilde{V})$ yields 
\begin{equation}\label{E:decompK}
\mathcal{K}(0)_{p} = \left( S_{\sup}^{s-p}(\widetilde{V}) \otimes \Lambda_{\sup}^{p}(\widetilde{V})\right) \oplus 
\left(e_{k}S_{\sup}^{s-p-1}(\widetilde{V}) \otimes \Lambda_{\sup}^{p}(\widetilde{V}) \right).
\end{equation}  However, by the definition of the differential on $\mathcal{K}(0)$ one sees that in fact this is a 
decomposition into subcomplexes.  Thus it suffices to show each direct summand is an exact complex.  

The first summand is equal to $\mathcal{K}(\widetilde{V})$ and so is exact since $\dim_{\C} \widetilde{V}_{\1}=k-1$ and the 
inductive assumption in our induction on $k$ applies. The second summand is isomorphic to $\mathcal{K}(\widetilde{V})$ 
(with $s$ replaced by $s-1$) via the map given by left multiplication by $e_{k}.$  Again by the inductive assumption 
in our induction on $k$ it follows that this complex is exact.  Since both summands are exact, $\mathcal{K}(0)$ is 
exact and thus the base case in our induction on $n$ is proven.

Now consider $n>0.$ Let us define a linear map 
\[
\texttt{R}_{e_{k}}:  \left[\mathcal{K}(n-1)/\mathcal{K}(n-2) \right]_{p-1} \to  \left[ \mathcal{K}(n)/\mathcal{K}(n-1) \right]_{p}
\] by right multiplication by $e_{k}:$ 
\[
a \otimes e_{\bj} \mapsto a \otimes (e_{\bj}\wedge e_{k}).
\]  Since $\left[\mathcal{K}(n-1)/\mathcal{K}(n-2) \right]_{p-1}$ is spanned by the images of the basis elements of 
$\mathcal{K}_{p-1}$ which have total degree $p-1$ and $e_{k}$-degree $n-1,$ and since $\left[ \mathcal{K}(n)/\mathcal{K}(n-1) 
\right]_{p}$ is spanned by the images of the basis elements of $\mathcal{K}_{p}$ which have total degree $p$ and 
$e_{k}$-degree $n,$  to see that $\texttt{R}_{e_{k}}$ defines a vector space isomorphism for each $p.$  
Furthermore, one can directly verify that $\texttt{R}_{e_{k}}$ defines a map of chain complexes. Let us point out that the differential for $\mathcal{K}(n)/\mathcal{K}(n-1)$ appears to have one extra summand, but that term is in fact zero in the quotient. Thus for all $p$, $\texttt{R}_{e_{k}}$ induces an isomorphism
\begin{equation}\label{E:inductiveiso}
\HH_{p-1}\left( \mathcal{K}(n-1)/\mathcal{K}(n-2) \right) \cong \HH_{p}\left( \mathcal{K}(n)/\mathcal{K}(n-1) \right).
\end{equation}
However, by the inductive assumption in our induction on $n,$ the left hand side of \eqref{E:inductiveiso} is zero for all $p.$  Hence, the right hand side is zero for all $p.$  

Therefore, by the induction on $n$ one has $E_{n,t}^{1} = 0.$  But, as explained above, by \eqref{E:spectral} one then has that $\HH_{\bullet}(\mathcal{K})=0$ whenever $\dim_{\C} V_{\1}=k.$  This proves the inductive step for the induction on $k$ and, hence, proves the result in general.
\end{proof}

We remark that our use of the map $\texttt{R}_{e_{k}}$ is in the spirit of the proof given in \cite[Section II.10]{knappvogan} for the exactness for the ordinary Koszul complex 

\subsection{}\label{SS:projresolution}  In this section we construct an explicit projective resolution in 
$\mathcal{F}$ for any object $M$ in $\mathcal{F}.$  The construction parallels the approach for 
relative cohomology of Lie algebras \cite{kumar}.  

Let $\fg$ be a Lie superalgebra and let $\ft$ be a Lie subsuperalgbra.  Furthermore, let $U=U(\fg)$ and $T=U(\ft)$ denote their respective enveloping superalgebras.  Define a chain complex as follows.  For $p \geq 0,$ set 
\begin{equation}\label{E:cochains}
D_{p} := U \otimes_T \Lambda_{\sup}^p(\fg/\ft ).
\end{equation} 
Define the complex 
\begin{equation}\label{E:complex}
\dotsb \xrightarrow{\partial_{3}} D_{2} \xrightarrow{\partial_{2}} D_{1} \xrightarrow{\partial_{1}} 
D_{0} \xrightarrow{\partial_{0}=\varepsilon \bar{\otimes} 1} \C \to 0
\end{equation}  
as follows. Here $\varepsilon \bar{{\otimes}} 1 : U \otimes_{T} \Lambda_{\sup}^{0}(\fg /\ft ) =  U \otimes_{T} \C  \to \C$ 
denotes the map given by $u \otimes a \mapsto \varepsilon(u)a$ where  $\varepsilon: U \to \C$ is the counit.  
For $p >0,$ the map 
\[
\partial_{p}:D_{p}\to D_{p-1}
\] is given by 
\begin{align*}
\partial_{p}(a \otimes y_1 \wedge ... \wedge y_p) &=
  \sum_{i<j} (-1)^{i+j+\varepsilon_{i,j}} a \otimes [y_i,y_j] \wedge y_1 \wedge \dotsb  \wedge \hat{y}_i \wedge 
\dotsb  \wedge \hat{y}_j \wedge \dotsb  \wedge y_p \\
 & + \sum_i (-1)^{i+1+\gamma_i} ay_i \otimes y_1 \wedge \dotsb  \wedge \hat{y}_i \wedge \dotsb  \wedge y_p
\end{align*}
where $a \in U$ and $y_{1}, \dotsc , y_{p} \in \fg /\ft $ are assumed to be homogeneous elements and $\varepsilon_{i,j}, \gamma_{i} \in \Z_2$ are given by the rule of signs for the $\Z_{2}$-grading:
\begin{align*}
\varepsilon_{i.j} &= \p{y}_i (\p{y}_1 + \dotsb + \p{y}_{i-1}) +  \p{y}_j (\p{y}_1 + \dotsb + \p{y}_{j-1}+\p{y}_i),\\
\gamma_i &=   \p{y}_i (\p{y}_1 + \dotsb + \p{y}_{i-1}).
\end{align*}  One can directly verify that the maps $\partial_{p}$ are well-defined and 
that $\partial_{p} \circ \partial_{p+1}=0$ for all $p \geq 0.$  Furthermore, as the following lemma shows, this an exact complex.  Let $\mathcal{C}(\fg ,\ft )$ denote the full subcategory of all $\fg$-modules which are finitely semisimple when restricted to $\ft$.

\begin{prop}\label{P:exactness}  The complex of superspaces given in \eqref{E:complex} is a projective 
resolution of the trivial module ${\mathbb C}$ in $\mathcal{C}(\fg ,\ft )$.
\end{prop}

\begin{proof} First observe that $D_{p}$ are projective modules in $\mathcal{C}(\fg ,\ft )$ for $p\geq 0$ 
\cite[Proposition 2.4.1]{BKN1}. 

In order to prove exactness we argue as in the proof of \cite[Proposition 3.1.5]{kumar}.  
Let $\{U_{r} \mid r \in \Z \}$ be the standard filtration of $U$ where $U_{r}$ is spanned by monomials 
in the PBW basis of total degree less than or equal to $r.$  We define a filtration on the complex, 
\[
D(0) \subseteq D(1) \subseteq D(2) \subseteq \dotsb 
\]
by setting $D(s)_{p} = U_{s-p} \otimes \Lambda_{\sup}^{p}(\fg/\ft)$.  For all $s \geq 0,$ set $D(s)_{-1}=\C.$ Note that indeed $D(s)$ is a subcomplex of $D$ for all $s \geq 0.$

We now consider the spectral sequence induced by this filtration.  Namely, 
\[
E_{s,t}^{1} = \HH_{s+t}\left(D(s)/D(s-1) \right) \implies  \HH_{s+t} \left(D_{\bullet} \right).
\] By the PBW theorem for Lie superalgebras one has that 
\[
\left[ D(s)/D(s-1)\right]_{p} \cong S_{\text{sup}}^{s-p}\left(\fg/\ft \right) \otimes \Lambda_{\sup}^{p}\left(\fg /\ft  \right).
\] The induced differential 
\[
\overline{\partial}_{p}: \left[ D(s)/D(s-1)\right]_{p} \to \left[ D(s)/D(s-1)\right]_{p-1}
\] is given by 
\[
a \otimes y_{1} \wedge \dotsb \wedge y_{p} \mapsto \sum_{i=1}^{p} (-1)^{i+1+\gamma_{i}} ay_{i} \otimes y_{1} \wedge \dotsb \wedge \widehat{y}_{i}\wedge \dotsb \wedge y_{p},
\] where $\gamma_{i}= \p{y}_{i}(\p{y}_{1}+\dotsb + \p{y}_{i-1}) \in  \Z_2.$

However, this is precisely the complex considered in Proposition~\ref{P:superkoszul}.  
Since that complex is exact, one has that $E_{s,t}^{1}=0$ and, hence,
$\HH_{\bullet}(D_{\bullet})=0.$  Therefore, the complex in \eqref{E:complex} is exact.
\end{proof}

\subsection{Complexity for ${\mathcal F}$} Let ${\mathcal V}=\{{V}_{t} \mid t\in {\mathbb N}\}=
\{{V}_{\bullet}\}$ be a sequence of finite dimensional ${\mathbb C}$-vector spaces. The 
{\it rate of growth} of ${\mathcal V}$, $r({\mathcal V})$, is the smallest positive integer $c$ such that 
there exists a constant $C>0$ with $\dim_{\C } V_{t}\leq C\cdot t^{c-1}$ for all $t$. If no such integer exists then ${\mathcal V}$ is said to have infinite rate of 
growth. Let $M\in {\mathcal F}$ and $P_{\bullet}\twoheadrightarrow M$ be a minimal projective resolution for $M$. Following 
Alperin \cite{Al}, we define the {\it complexity of $M$} to be $r(\{P_{n} \mid n=0,1,2,\dots \})$.\footnote{Several years 
after his definition of the complexity, Alperin wrote \cite{Al2}:  
\begin{quotation} It is quite surprising that concentrating on the simplest possible invariant, the size, 
of resolutions should lead to so much mathematics and applications to other sorts of questions in 
representation theory and mathematics, as well as to a host of problems of a homological nature in 
modular representation theory..... 
\end{quotation}} In the following 
theorem we prove that the complexity is finite for all $M\in {\mathcal F}$. 

\begin{theorem}\label{T:projectiveres}  Let $M$ be an object of $\mathcal{F}.$  Then $c_{\mathcal F}(M)\leq \dim_{\C} {\mathfrak g}_{\1}$. 
\end{theorem}
\begin{proof} We will show that $M$ admits a projective resolution in $\mathcal{F}$ with growth rate equal to $\dim_{\C} \fg_{\1}$. 
The result will follow because a minimal projective resolution of $M$ is a direct summand of any projective resolution 
of $M$. 

Consider \eqref{E:complex} when $\ft = \fg_{\0}.$  Applying the exact functor 
$- \otimes M$ to the exact sequence \eqref{E:complex} yields an exact sequence $D({M})_{\bullet}$ with 
\[
D({M})_{p} = \left( U(\fg ) \otimes_{U(\fg_{\0})}\Lambda_{\sup}^p(\fg_{\1} )\right) \otimes 
M \cong U(\fg ) \otimes_{U(\fg_{\0})}  \left(\Lambda_{\sup}^p(\fg_{\1} ) \otimes M\right),
\] where the last isomorphism is by the tensor identity.  
By \cite[Proposition 2.4.1]{BKN1} $D({M})_{p}$ is a projective object in $\mathcal{F}.$  Thus we 
have constructed a projective resolution of $M$ in $\mathcal{F}.$

Now we observe that as vector spaces one has 
\[
D({M})_{p} \cong S_{\sup}(\fg_{\1}) \otimes \left(\Lambda_{\sup}^p(\fg_{\1} ) \otimes M\right) 
\cong \Lambda(\fg_{\1}) \otimes S^p(\fg_{\1} ) \otimes M.
\]  Since $\Lambda(\fg_{\1})$ and $M$ are finite dimensional, this complex has the same rate of growth as 
$S^{\bullet}(\fg_{\1}).$  That is, $r(D({M})_{\bullet})$ equals $\dim_{\C} \fg_{\1}.$
\end{proof}

\subsection{Examples} In this section we will provide a description of the minimal projective 
resolutions for the principal block of ${\mathcal F}$ when ${\mathfrak g}={\mathfrak q}(1)$ or 
$\mathfrak{gl}(1|1)$. 

In the case when ${\mathfrak g}={\mathfrak q}(1)$ the trivial module is the only simple module 
in the principal block. The projective cover $P({\mathbb C})$ of ${\mathbb C}$ has two composition factors. 
Therefore, the minimal projective resolution of ${\mathbb C}$ is given by 
$$\dots \rightarrow P({\mathbb C}) \rightarrow P({\mathbb C}) \rightarrow P({\mathbb C}) \rightarrow {\mathbb C}\rightarrow 0.$$ 
Thus $c_{\mathcal F}({\mathbb C})=1$. In this case the Krull dimension of the relative cohomology ring 
$ \HH^{\bullet}({\mathfrak g},{\mathfrak g}_{\0};\linebreak[0]{\mathbb C})$ is $1$ \cite[Table 1]{BKN1}.

Now let ${\mathfrak g}=\mathfrak{gl}(1|1)$. The simple modules in the principal block are 
one dimensional and indexed by $L(\lambda \,| -\lambda)$ where $\lambda\in {\mathbb Z}$. 
The projective cover $P(\lambda\,|-\lambda)$ is four dimensional with three radical layers. 
The head and socle of $P(\lambda\,|-\lambda)$ are both isomorphic to  
$L(\lambda\,|-\lambda)$ and the second layer is isomorphic to $L(\lambda+1\,|-\lambda-1)\oplus 
L(\lambda-1\,|-\lambda+1)$. The minimal projective resolution of the trivial module $L(0\,|\, 0)$ is given by 
$$\dots \rightarrow P(2\,| -2)\oplus P(0\,|\, 0) \oplus P(-2\,|\, 2) \rightarrow P(1\,| -1)\oplus P(-1\,|\, 1) 
\rightarrow P(0\,|\, 0) \rightarrow L(0\,|\, 0) \rightarrow 0.$$ 
Therefore, $\dim_{\C} P_{n}=4(n+1)$ and $c_{\mathcal F}(L(0\,|\, 0))=2$. In fact, one can 
easily show that $c_{\mathcal F}(L(\lambda\,| -\lambda))=2$ for all $\lambda\in {\mathbb Z}$. 
But the atypicality of every simple module in the principal block is one, the Krull dimension of $ 
\HH^{\bullet}({\mathfrak g},{\mathfrak g}_{\0};{\mathbb C})$ \cite[Table 1]{BKN1}. This example demonstrates that the relative 
cohomology ring is not large enough to measure the complexity of modules in ${\mathcal F}$. Another 
example which also illustrates this point is the Borel subalgebra of $\mathfrak{gl}(1|1)$ given 
by ${\mathfrak b}={\mathfrak g}_{0}\oplus {\mathfrak g}_{1}$ (cf.\ Section \ref{S:typeI}). 
Note that in all three of these examples, $c_{\mathcal F}({\mathbb C})=\dim_{\C} {\mathfrak g}_{\1}$. For 
higher rank algebras our preliminary calculations suggest that the complexity of the 
trivial module will usually be strictly less than $\dim_{\C} {\mathfrak g}_{\1}$.

\subsection{Projective and Periodic Modules} In the category ${\mathcal F}$ we can characterize 
the property of being projective via the complexity. 

\begin{cor} Let $M\in {\mathcal F}$. Then $c_{{\mathcal F}}(M)=0$ if and only if $M$ is projective. 
\end{cor} 

\begin{proof} Clearly if $M$ is projective then $c_{{\mathcal F}}(M)=0$. On the other hand if 
$c_{{\mathcal F}}(M)=0$ then $M$ admits a finite projective resolution $0\rightarrow 
P_{n}\rightarrow P_{n-1} 
\rightarrow \dotsb \rightarrow P_{0}\rightarrow M \rightarrow 0$. By Proposition~\ref{P:self-injective}, 
$P_{n}$ is injective and this resolution will split. Therefore, $M$ is projective. 
\end{proof} 

A module $M$ is {\it periodic} if it is non-projective and admits a periodic projective resolution. 
For a finite dimensional cocommutative Hopf algebra $A$, a module $M$ is periodic if and only if 
$c_{A}(M)=1$. Observe for our category ${\mathcal F}$, if $M$ is periodic, then $c_{\mathcal F}(M)=1$. 
However, it is not true in general that $c_{\mathcal F}(M)=1$ implies that $M$ is periodic. Consider 
${\mathfrak g}=\mathfrak{gl}(1|1)$.  Given the structure of the projective indecomposable modules given in the previous section it is straightforward to verify that the Kac module $K(0\,|\, 0)$ (cf.\ Section~\ref{SS:Kacasses}) admits a minimal projective resolution of the form
$$\dots \rightarrow P(2\,| -2) \rightarrow P(1\,| -1) \rightarrow P(0\,|\, 0) \rightarrow K(0\,|\, 0) \rightarrow 0.$$ 
Here, $\dim_{\C} P(\lambda\,| -\lambda)=4$ for all $\lambda\in {\mathbb Z}$, thus $c_{\mathcal F}(K(0\,|\, 0))=1$. 

\subsection{Interpreting complexity via relative cohomology}\label{SS:complexityrelcoho} Recall that $\mathcal{C}(\fg,\fg_{\0})$ is the category of all $\fg$-modules which are finitely semisimple upon restriction to $\fg_{\0}$.  We can now relate the complexity of 
a module $M$ to the rate of growth of a sequence of $\Ext$-groups in the category ${\mathcal C}({\mathfrak g},{\mathfrak g}_{0})$. Note that in the formulation below one needs to take the direct sum of all simple modules in 
${\mathcal F}$. This module is infinite dimensional and is an object in ${\mathcal C}({\mathfrak g},
{\mathfrak g}_{0})$ but not in ${\mathcal F}$.

Here and elsewhere we write $\Ext^{\bullet}_{(\fg ,\fg_{\0})}(M,N)$ for relative cohomology for the pair $(\fg ,\fg_{\0}).$   When $M,N \in \mathcal{C}(\fg ,\fg_{\0})$ we freely make use of the isomorphism  \cite[Corollary 2.4.2, Theorem 2.5.1]{BKN1}
\[
\Ext^{\bullet}_{\mathcal{C}(\fg , \fg_{\0})}(M,N)  \cong \Ext^{\bullet}_{(\fg , \fg_{\0})}(M,N), 
\] and for $M,N \in \mathcal{F}$,
\[
\Ext^{\bullet}_{\mathcal{F}}(M,N) \cong \Ext^{\bullet}_{(\fg , \fg_{\0})}(M,N).
\]  Furthermore since the indecomposible projectives in $\mathcal{C}(\fg,\fg_{\0})$ are finite dimensional, the minimal projective resolution of an object $M$ in $\mathcal{F}$ is also the minimal projective resolution in $\mathcal{C}(\fg ,\fg_{\0})$.  

Let us also remark that since our convention is to allow all (not only graded) module homomorphisms, for a simple module $S$ it can happen that $\dim_{\C} \Hom_{\mathcal{F}}(S,S)$ is either one or two.  Set 
\begin{equation}\label{E:Homs}
\kappa_{S} = \dim_{\C} \Hom_{\mathcal{F}}(S,S)
\end{equation}
for each simple module $S$ in $\mathcal{F}$.

\begin{prop}
Let $M\in\F$ and let $P_{\bullet}\twoheadrightarrow M$ be a minimal projective resolution. Then
$$
c_{\F}(M):=r(P_{\bullet}) = r\left(\Ext^{\bullet}_{(\g,\g_{\0})}(M,\bigoplus  S^{\dim_{\C} P(S)})\right)
$$
where the sum is over all simple modules $S$ in $\F$, and $P(S)$ is the projective cover of $S$.
\end{prop}

\begin{proof} Using the fact that $P_{\bullet}\twoheadrightarrow M$ is a minimal projective resolution, 
we have 
$$\Ext^{n}_{\mathcal{C}(\g,\g_{\0})}\left( M,\bigoplus  S^{\dim_{\C} P(S)}\right) = \Hom_{\mathcal{C}(\g,\g_{\0})}\left( P_{n},\bigoplus  
S^{\dim_{\C} P(S)}\right).$$ 
Therefore, 
$$
\begin{aligned}
\dim_{\C} \Ext^{n}_{(\g,\g_{\0})}\left( M,\bigoplus  S^{\dim_{\C} P(S)}\right) &= \sum_{S}\dim_{\C} P(S)\, \dim_{\C} \Hom_{\mathcal{C}(\g,\g_{\0})}(P_{n},S) \\
 &= \sum_{S}\kappa_{S}\dim_{\C} P(S)\, [P_{n}:P(S)] \\ 
 &\leq \sum_{S}2\dim_{\C} P(S)\, [P_{n}:P(S)] \\
 &= 2\dim_{\C} P_{n} .
\end{aligned}
$$  This proves the desired result.
\end{proof}

\subsection{Blocks with finitely many simples}\label{SS:blocks} If ${\mathcal B}$ is a block of ${\mathcal F}$ 
with only finitely many simple modules, then one can measure the rate of growth of projective resolutions 
with the relative cohomology. For $M\in {\mathcal F}$, let $V_{({\mathfrak g},{\mathfrak g}_{\0})}(M)$ be the relative 
support variety of $M$ as defined in \cite[Section 6.1]{BKN1}. The proofs for finite groups (cf.\ \cite[Sections 5.3, 5.7]{benson}) in 
conjunction with \cite[Theorem 2.5.3]{BKN1} can be applied in this situation to prove the 
following result.
 
\begin{theorem} Let ${\mathcal B}$ be a block of ${\mathcal F}$ with only finitely many simple modules. If $M$ 
is in ${\mathcal B}$ then the following statements hold.
\begin{itemize} 
\item[(a)] $c_{\F}(M)= r\left(\Ext^{\bullet}_{(\g,\g_{\0})}(M,\bigoplus  S)\right)$
where the sum is over all simple modules $S$ in $\mathcal{B}$. 
\item[(b)] $c_{\F}(M)= r\left(\Ext^{\bullet}_{(\g,\g_{\0})}(M,M)\right)$. 
\item[(c)] $c_{\F}(M)= \dim V_{({\mathfrak g},{\mathfrak g}_{\0})}(M)$. 
\end{itemize} 
\end{theorem} 

Say $\fg$ is a simple classical Lie superalgebra of type $A(m,n)$, $B(m,n)$, $C(n)$, $D(m,n)$, $D(2,1;\alpha)$, $F(4)$, $G(3)$, $P(n)$, $Q(n)$ (note that if ${\mathfrak g}$ is $P(n)$ or $Q(n)$, 
then necessarily $n > 1$). If a block of $\mathcal{F}$ contains an atypical simple module, then it has infinitely many simple modules as one 
can add integral multiples of some odd root to the highest weight to obtain other finite dimensional simple modules in the block.
On the other hand, if every simple module in the block is typical, then there are finitely many simple modules and the block is known to be semisimple.  However, the non-simple classical Lie superalgebras $\tilde{\fp}(1)$ and $\fq(1)$ admit atypical blocks which have finitely many simples.  It should be remarked that the detecting subalgebras introduced in \cite{BKN1} are essentially isomorphic to a direct product of Lie superalgebras isomorphic to $\fq(1)$.  Thus the above theorem applies in a setting which is already known to be of interest when studying cohomology and support varieties for Lie superalgebras.

\section{Type I Lie superalgebras} \label{S:typeI}

Recall that a Lie superalgebra is said to be of Type I if it admits a consistent $\Z$-grading concentrated 
in degrees $-1,$ $0,$ and $1.$  Otherwise it is of Type II.  The distinction 
between these two types was first made by Kac \cite{Kac1}.

\subsection{}\label{SS:Kacasses} Let ${\fg}=\fg_{-1}\oplus {\fg}_{0}\oplus 
{\fg}_{1}$ be a Type I classical Lie superalgebra (cf.\ [BKN2]).    Note that the Lie superalgebra $\gl (m|n)$ 
and the simple Lie superalgebras of types $A(m,n)$, $C(n)$ and $P(n)$ are all of Type I (cf.\ \cite[Section 2]{Kac1}).
Let 
$$
\fp^{+}= \fg_{0} \oplus \fg_{1} \qquad \text{and} \qquad  \fp^{-}= \fg_{0} \oplus \fg_{-1} .
$$
We observe that, because of the $\Z$-grading, $\fg_{1}$ is an ideal of $\fp^{+}$ and $\fg_{-1}$ is an ideal of $\fp^{-}$.
Fix a Cartan subalgebra $\fh  \subseteq \fg_{0}$ and Borel subalgebras $\fb_{0} \subseteq \fg_{0}$ and 
$\fb \subseteq \fg$ such that $\fh \subseteq \fb_{0}$ and \begin{equation}\label{E:borelequality}
\fb = \fb_{0} \oplus \fg_{1}.
\end{equation}

Let $X^{+}_{0} \subseteq \fh^{*}$ denote the parameterizing set of highest weights for the simple 
finite dimensional $\fg_{0}$-modules with respect to the pair $(\fh, \mathfrak{b}_{0})$.  Given 
$\lambda \in X^{+}_{0},$ let $L_{0}(\lambda)$ denote the simple $\fg_{0}$-module of highest weight 
$\lambda$ (concentrated in degree $\0 $).  We view $L_{0}(\lambda)$ as a simple $\fp^{\pm}$-module via 
inflation through the canonical quotient $\fp^{\pm} \twoheadrightarrow \fg_{0}.$  Let
\[
K(\lambda)=U(\fg)\otimes_{U(\fp^{+})} L_{0}(\lambda) \qquad \text{and} \qquad K^{-}(\lambda) = 
\Hom_{U(\fp^{-})}\left(U(\fg), L_{0}(\lambda) \right)
\] be the \emph{Kac module} and the \emph{dual Kac module}, respectively.  
A standard argument using \eqref{E:borelequality} and the fact that $\fg_{1}$ is an 
ideal of $\fp^{+}$ shows that $K(\lambda)$ is the universal highest weight module 
in $\mathcal{F}$ of highest weight $\lambda$.  By highest weight arguments, $K(\lambda)$ has 
simple head $L(\lambda)$ and the set 
\[
\setof{L(\lambda)}{\lambda \in X_{0}^{+}}
\] is a complete irredundant collection of simple objects in $\mathcal{F}.$  Let $P(\lambda)$ 
(resp.\ $I(\lambda)$) denote the projective cover (resp.\ injective envelope) in $\mathcal{F}$ for the 
simple ${\mathfrak g}$-module $L(\lambda)$.

\subsection{Cohomology and Support Varieties}\label{SS:cohom}
Observe that both $\fg_{1}$ and $\fg_{-1}$ are abelian Lie superalgebras and, consequently, 
\[
R_{\pm}=\HH^{\bullet} (\fg_{\pm 1}, \C ) = \HH^{\bullet}(\fg_{\pm 1}, \{0 \};\C) \cong S(\fg_{\pm 1}^{*})
\] as graded algebras.  Let $\mathcal{F}(\fg_{\pm 1})$ be the category of finite dimensional $\fg_{\pm 1}$-modules.  If $M$ is an object in $\mathcal{F}(\fg_{\pm 1})$, then one can define the $\fg_{\pm 1}$ \emph{support variety} of $M$ just as in \cite{BKN1}.  Namely, set 
\[
I_{M} = \left\{r \in R_{\pm} \mid r.m=0 \text{ for all } m \in \Ext_{\mathcal{F}(\fg_{\pm 1})}^{\bullet}(M,M) \right\}
\] and then the support variety of $M$ is 
\[
\V_{\fg_{\pm 1}}(M) = \operatorname{MaxSpec}\left(R_{\pm}/I_{M} \right).
\]  Since $\fg_{\pm 1}$ is abelian the arguments given in \cite[Section 5]{BKN1} for detecting subalgebras apply here as well and one has that $\V_{\fg_{\pm 1}}(M)$ is canonically isomorphic to the following rank variety: 
\[
\V_{\fg_{\pm 1}}^{\text{rank}}(M) =\left\{ x \in \fg_{\pm 1} \mid M \text{ is not projective as a $U(\langle x \rangle)$-module} \right\} \cup \{ 0 \}.
\]  Consequently, $\V_{\fg_{\pm 1}}(M)$ satisfies the various properties of a rank variety (e.g.\ it satisfies the tensor product rule and detects $\fg_{\pm 1}$ projectivity \cite[Sections 5, 6]{BKN1}). We will identify $\V_{\fg_{\pm 1}}(M)$ and $\V^{\text{rank}}_{\fg_{\pm 1}}(M)$ via this canonical isomorphism.

\subsection{Kac Filtrations} \label{SS:Kacfiltrations}

Given a module $M$ in $\mathcal{F},$ say $M$ admits a \emph{Kac} (resp.\ \emph{dual Kac}) \emph{filtration} if there is a filtration 
\[
M=M_{0} \supsetneq M_{1} \supsetneq M_{2} \supsetneq \dotsb \supsetneq M_{t} \supsetneq \{0\} 
\] as $\fg$-modules such that for each $k \geq 0$ such that $M_{k}\neq 0,$ one has $M_{k}/M_{k+1} \cong K(\mu)$ (resp.\ $K^{-}(\mu)$) for some $\mu \in X_{0}^{+}.$ A module $M$ is called \emph{tilting} if it admits both a Kac and a dual Kac filtration.

\begin{theorem}\label{T:Kacfiltrations}  Let ${\mathfrak g}$ be a Type I classical Lie superalgebra, 
and $M$ be a $\fg$-module.  Then the following are equivalent.
\begin{itemize}
\item[(a)] $M$ has a Kac filtration.
\item[(b)] $\Ext^{1}_{\F}(M, K^{-}(\mu))=0$ for all $\mu \in X^{+}(T)$.
\item[(c)] $\Ext^{1}_{\mathcal{F}(\fg_{-1})}(M, \C ) =0$.
\item[(d)] $\V_{\fg_{-1}}(M)=0$.
\end{itemize}
\end{theorem}

\begin{proof} The equivalence of (a) and (b) follows by standard arguments (cf.\ \cite[II Ch.\ 4]{jantzen}). 
We will next prove the equivalence of (b) and (c). First observe that 
\begin{equation}\label{E:iso1}
\Ext^{j}_{\mathcal{F}}(M,K^{-}(\mu)) \cong \Ext^{j}_{(\fg,\fg_{\0})}(M,K^{-}(\mu))\cong  
\Ext^{j}_{(\fp^{-},\fg_{\0})}\left(M,L_{0}(\mu) \right)
\end{equation}
for 
$j\geq 0$. Since $\fg_{-1}$ is an ideal of $\fp^{-},$ one can apply the Lyndon-Hochschild-Serre spectral sequence 
for the pair $(\fg_{-1}, \{0 \})$ in $(\fp^{-}, \fg_{\0})$ (cf.\ \cite[Theorem 6.5]{BorelWallach}): 
\[
E_{2}^{i,j}=\Ext^{i}_{(\fg_{\0},\fg_{\0})}(L_{0}(\mu)^{*},
\Ext^{j}_{(\fg_{-1},\{0\})}(M,\C))\Rightarrow   
\Ext^{i+j}_{(\fp^{-},\fg_{\0})}(M,L_{0}(\mu)).
\]  
Since $\Ext_{(\fg_{\0},\fg_{\0})}^{i}(-,-)$ calculates extensions in the category of 
$\fg_{\0}$-supermodules which are completely reducible, the higher $\Ext$s vanish and the spectral sequence collapses.  
Therefore, we have 
\begin{equation}\label{E:iso2}
\Hom_{\fg_{\0}}\left(L_{0}(\mu)^{*}, \Ext^{j}_{(\fg_{-1},\{0\})}(M,\C) \right) \cong   
\Ext^{j}_{(\fp^{-},\fg_{\0})}\left(M,L_{0}(\mu)\right),
\end{equation} for all $j\geq 0.$
Combining \eqref{E:iso1} and \eqref{E:iso2}, we obtain for $j=1$
\begin{equation}\label{E:iso3}
\Ext^{1}_{\mathcal{F}}(M,K^{-}(\mu)) \cong
\Hom_{\fg_{\0}}\left(L_{0}(\mu)^{*}, \Ext^{1}_{(\fg_{-1},\{0\})}(M,\C) \right).  
\end{equation} Now since $\Ext^{1}_{(\fg_{-1},\{0\})}(M,\C) = 0 $ if and only if $\Hom_{\fg_{\0}}\left( L_{0}(\mu)^{*}, \Ext^{1}_{(\fg_{-1},\{0\})}(M,\C) \right) = 0$ for all $\mu \in X^{+}(T)$, \eqref{E:iso3} implies (b) is equivalent to (c).


Since $U(\fg_{-1})$ is isomorphic to an exterior algebra whose only simple module is ${\mathbb C}$, 
it follows that $M$ is projective over ${\mathfrak g}_{-1}$ if and only if $\Ext^{1}_{\fg_{-1}}(M,\C )=0$.  
The equivalence of (c) and (d) now follows from the fact that the $\fg_{-1}$-support variety detects $\fg_{-1}$-projectivity.
\end{proof}

Note that one can formulate an analogous ``dual'' theorem for dual Kac filtrations.  

\begin{theorem}\label{T:dualKacfiltrations}  Let ${\mathfrak g}$ be a Type I classical Lie superalgebra, 
and $M$ be a $\fg$-module.  Then the following are equivalent.
\begin{itemize}
\item[(a)] $M$ has a dual Kac filtration.
\item[(b)] $\Ext^{1}_{\F}(K^{+}(\mu),M)=0$ for all $\mu \in X^{+}(T)$.
\item[(c)] $\Ext^{1}_{\mathcal{F}(\fg_{1})}(\C,M) =0$.
\item[(d)] $\V_{\fg_{1}}(M)=\{0\}$.
\end{itemize}
\end{theorem}

From Theorem~\ref{T:Kacfiltrations} and Theorem~\ref{T:dualKacfiltrations} one sees that the definition of tilting used here is equivalent to the one used by Brundan \cite{brundan2}.  We can also deduce the following result. 

\begin{cor}\label{C:tilting} Let ${\mathfrak g}$ be a Type I classical 
Lie superalgebra. A $\fg$-module $M$  is tilting if and only if $\V_{\fg_{1}}(M) = \V_{\fg_{-1}}(M) =\{0\}$.
\end{cor}

\subsection{Duality}\label{SS:duality} In many instances Type I Lie superalgebras admit an 
antiautomorphism $\tau : \fg  \to \fg$ such that $\tau(\fg_{i})=\fg_{-i}$ ($i \in \Z$) and, when 
restricted to $\fg_{0}$, $\tau$ coincides with the Chevalley antiautomorphism. For Type I simple 
classical Lie superalgebras of type A and C, such an antiautomorphism exists by the proof of \cite[Proposition 2.5.3]{Kac1}. Namely, these algebras are contragradient and can be given by Chevalley generators and relations.  On these generators the antiautomorphism is given by $e_{i}\mapsto f_{i}$, $f_{i} \mapsto (-1)^{\p{e}_{i}}e_{i}$, and $h_{i}\mapsto h_{i}$.   For $\gl(m|n)$ one can take $\tau$ to be the supertranspose map \cite[Section 4-a]{brundan1}. 
If ${\mathfrak g}$ admits such an antiautomorphism then we say that ${\mathfrak g}$ has a 
\emph{strong duality}.

Given a $\fg$-module $M$ we write $M^{\tau}$ for the \emph{transpose dual} of M; that is, $M = M^{*}$ as a $\Z_{2}$-graded vector space and the action of $\fg$ is given by $(x.f)(m) = (-1)^{\p{x}\; \p{f}}f(\tau(x)m)$ for homogeneous $x \in \fg$, $f \in M^{*}$, and $m \in M$. If $S$ is a simple $\fg$-module then by character considerations we have $S^{\tau}\cong S$.

Now suppose that ${\mathfrak g}$ is a Type I Lie superalgebra with a strong duality $\tau$. 
Observe that $\V_{{\mathfrak g}_{-1}}({\mathbb C})\cong {\mathfrak g}_{-1}=\tau({\mathfrak g}_{1})\cong 
\tau(\V_{{\mathfrak g}_{1}}({\mathbb C}))$. Moreover, this isomorphism is compatible with 
the definition of support varieties.  Namely, if $M \in \mathcal{F}(\fg_{\mp 1})$ (cf.\ Section~\ref{SS:cohom}), then $M^{\tau}$ is a $\fg_{\pm 1}$-module and one has
\begin{equation}\label{E:supportduality} 
\V_{\fg_{\pm 1}}(M^{\tau})=\tau(\V_{\fg_{\mp1}}(M)). 
\end{equation} 

\subsection{Projective and Tilting are Equivalent for Type I}\label{SS:tilting}

The following result is inspired by a result by Cline, Parshall and Scott on injective 
modules for infinitesimal algebraic groups (cf.\ \cite[(2.1) Theorem]{CPS:85}). 

\begin{theorem}\label{T:projective} Let ${\mathfrak g}$ be a Type I classical Lie superalgebra and 
$M\in {\mathcal F}$. Then  
\begin{itemize} 
\item[(a)] $M$ is projective if and only if $\V_{\fg_{1}}(M) = \V_{\fg_{-1}}(M) =\{0\}$.
\item[(b)] Suppose that $\fg$ admits a strong duality $\tau$. If $M^{\tau}\cong M$, then 
$M$ is projective if and only if $\V_{\fg_{1}}(M)= \{0\}$.
\end{itemize} 
\end{theorem}

\begin{proof} (a) Suppose that $M$ is projective. Since restriction is an exact functor it follows that $M$ is a projective module for $\fg_{\pm 1}$ and, hence, $\V_{\fg_{\pm  1}}(M)=0$.

On the other hand, suppose that $\V_{\fg_{1}}(M) = \V_{\fg_{-1}}(M) = 0$.  We first observe that $M$ is projective in ${\mathcal F}$ if and only if 
$\Ext^{n}_{({\mathfrak g},{\mathfrak g}_{\0})}(M,S)=0$ for all $n\geq 1$ and all simple $\fg$-modules $S$.  By 
Theorems~\ref{T:Kacfiltrations} and \ref{T:dualKacfiltrations}, $M$ has a Kac 
filtration and $M$ has a dual Kac filtration. So $M^{*}$ has a filtration with duals 
of dual Kac modules which are of the form $\widehat{K}(\mu)$ where 
$\widehat{K}(\mu)=U({\mathfrak g})\otimes_{U({\mathfrak p}^{-})}L_{0}(\mu)$. 
Consequently, $M\otimes M^{*}$ has a filtration 
with sections of the form $K(\lambda)\otimes \widehat{K}(\mu)$. But for any simple $\fg$-module $S$ we have 
\begin{eqnarray*} 
\Ext^{n}_{({\mathfrak g},{\mathfrak g}_{\0})}(K(\lambda)\otimes \widehat{K}(\mu),S)
&\cong& \Ext^{n}_{({\mathfrak p}^{+},{\mathfrak g}_{\0})}(\widehat{K}(\mu),S\otimes L_{0}(\lambda)^{*})\\
&\cong& \Ext^{n}_{({\mathfrak p}^{+},{\mathfrak g}_{\0})}(U({\mathfrak p}_{+})
\otimes_{U({\mathfrak g}_{0})}L(\mu),S\otimes L_{0}(\lambda)^{*})\\
&\cong& \Ext^{n}_{({\mathfrak g}_{\0},{\mathfrak g}_{\0})}(L(\mu),S\otimes L_{0}(\lambda)^{*})\\
&=& 0
\end{eqnarray*} 
for $n\geq 1$. It follows that $K(\lambda)\otimes \widehat{K}(\mu)$ is projective in ${\mathcal F}$, 
thus $M\otimes M^{*}$ is projective in ${\mathcal F}$. Now we continue to follow the argument 
given in \cite{CPS:85}: $M\otimes M^{*}\cong \End_{{\mathbb C}}(M)$ and there exists 
a split $U({\mathfrak g})$ surjective map $\End_{{\mathbb C}}(M)\otimes M\rightarrow M$. Thus 
$M$ is a summand of a projective module and is therefore projective.  

For part (b) if $M\cong M^{\tau}$ then $\V_{\fg_{-1}}(M) \cong \V_{\fg_{1}}(M^{\tau}) = \V_{\fg_{1}}(M)$ 
by \eqref{E:supportduality}. Therefore, one can apply part (a) to obtain the result. 
\end{proof} 

One can now apply Corollary~\ref{C:tilting} and Theorem~\ref{T:projective} to prove the equivalence 
of tilting and projective modules for Type I superalgebras. 

\begin{cor}\label{C:BKNproj3} Let ${\mathfrak g}$ be a Type I classical Lie superalgebra. 
A $\fg$-module $M\in\F$ is projective if and only if $M$ is tilting. 
\end{cor}

We remark that the above result is implicit in \cite{brundan2}.  Namely, in the setting of $\Z$-graded Lie superalgebras which satisfy certain mild conditions Brundan provides a duality between projective and tilting modules.  See in particular \cite[Example 7.5]{brundan2}.  Note that our approach also allows one to rederive \cite[Lemma 4.2]{brundan2}.

\subsection{Connection with the Duflo-Serganova associated varieties}\label{SS:detectingprojectives}  
In this section we demonstrate how our results interface with theorems proved by Duflo and 
Serganova \cite{dufloserganova} on their associated varieties. Consider the subvariety of $\fg_{\1}$ given by
\[
\X = \left\{ x \in \fg_{\1} \mid x^{2} = [x,x]/2 =0 \right\}.
\]  For $M\in \F$, Duflo and Serganova define the following associated variety 
\[
\X_{M} = \left\{x \in \X \mid  \Ker (x) / \operatorname{Im}(x) \neq 0 \right\}
\] where $x$ is viewed as a linear map $M \to M$ satisfying $x^{2}=0$.  
Using the description of the representations of $U(\langle x \rangle)$ as described in \cite[Proposition 5.2.1]{BKN1}, one can verify that the condition 
$\Ker (x) / \operatorname{Im}(x) \neq 0$ is equivalent to the condition that $M$ is not projective as a 
$U(\langle x  \rangle)$-module.  That is, 
\begin{equation}\label{E:XMdef}
\X_{M} =\left\{ x \in \X  \mid M \text{ is not projective as a $U(\langle x \rangle)$-module} \right\} \cup \{0 \}.
\end{equation}
Note that for Type I Lie superalgebras, 
$\fg_{\pm 1}$ is abelian, so $\fg_{\pm 1} \subset \X$. Using the rank variety description we immediately deduce that 
\begin{equation}\label{relationship} 
\V_{\fg_{-1}}(M)\cup \V_{\fg_{1}}(M) \subseteq \X_{M},
\end{equation} 
\begin{equation}\label{relationship2} 
\V_{\fg_{\pm 1}}(M)= \X_{M}\cap {\mathfrak g}_{\pm 1}.
\end{equation} 

We can now recover the following theorem of Duflo and Serganova \cite[Theorem 3.4]{dufloserganova} 
(in the case when ${\mathfrak g}$ is a classical Type I Lie superalgebra) which shows that $\X_{M}$ detects projectivity.  We remark that the condition assumed in \cite{dufloserganova} that $\X$ spans 
${\mathfrak g}_{\1}$ automatically holds in our setting.  

\begin{theorem}\label{T:DSproj} Let ${\mathfrak g}$ be a classical Type I Lie superalgebra and let $M$ be in $\F$. Then $\X_{M}=0$ if and only if $M$ is projective.
\end{theorem}

\begin{proof} Suppose that $M\in {\mathcal F}$ and $\X_{M}=0$. Then by \eqref{relationship2} 
we have $\V_{\fg_{\pm 1}}(M)=\{0\}$. Now apply 
Theorem~\ref{T:projective} to conclude that $M$ is projective. 

Conversely, suppose that $M$ is projective. Then by Theorem~\ref{T:projective} $\V_{\fg_{\pm 1}}(M)=\{0\}$. 
Moreover, since $M$ is projective it is in fact a ${\mathbb Z}$-graded module (for example, as argued in the proof of \cite[Proposition 3.4.1]{BKN2}).  Fix $a \in \mathbb{R}$ with $a > 1$.  We can then define an action of $\Z$ (written multiplicatively with fixed generator $t$) on $\fg$ (resp.\ $M$) by $t.x = a^{k}x$ for $x \in \fg_{k}$ (resp.\ $t.m = a^{k}m$ for $m \in M_{k}$), where $k \in \Z$.  We note that $t.(xm) = (t.x)(t.m)$ for all $x \in \fg$ and $m \in M$.  

Now suppose that $x\in \X_{M}$ and $x \neq 0$.  Then by the rank variety description of $\X_{M}$ and \cite[Proposition 5.2.1]{BKN1} it follows that when $M$ is considered as an $\langle x \rangle$-module a trivial module appears as a direct summand; say it is spanned by $m \in M$.  We then check that $t.m$ spans a trivial direct summand of $M$ as a $\langle t.x \rangle$-module.  Hence $t.x \in \X_{M}$. Thus $\X_{M}$ is stable under this action of $\Z$.  Now let $x \in \X_{M}$  and write $x = x_{-1} + x_{1}$ with $x_{\pm 1} \in \fg_{\pm 1}$.  
Since $\X_{M}$ is stable under the action of $t$ it follows that $t^{n}x = a^{-n}x_{-1} + a^{n}x_{1} \in \X_{M}$ for all $n > 0$.  Since $\X_{M}$ is also conical we can scale by $a^{-n}$ and see that $a^{-2n}x_{-1}+x_{1}\in \X_{M}$ 
for all $n > 1$.  However, as $\X_{M}$ is closed it follows by letting $n$ go to infinity that $x_{1} \in \X_{M}$.  
If $x_{1} \neq 0,$ then $ x_{1} \in \V_{\fg_{1}}(M) \neq 0$ by \eqref{relationship2}.  
On the other hand, if $x_{1}=0,$ then $x = x_{-1} \in \X_{M}$ and then $x_{-1} \in \V_{\fg_{-1}}(M) \neq 0$.
\end{proof}

\subsection{Refinements of the Projectivity Test} In this section we show that 
we can further improve on the Duflo-Serganova criterion for projectivity in ${\mathcal F}$ by 
showing it is sufficient to test projectivity over a ``minimal'' collection of elements. 
In all known instances this minimal collection is finite and, indeed, often a single element of $\fg_{1}$.  For restricted Lie algebras, analogous projectivity tests were proved by Friedlander and Parshall \cite[(2.4) Proposition]{FPa}.

Let $G_{\0}$ be the connected reductive algebraic group such that $\operatorname{Lie}(G_{\0}) = \fg_{\0}$ 
and the adjoint action of $G_{\0}$ on $\fg_{\pm 1}$ differentiates to the adjoint action of 
$\fg_{\0}$ on $\fg_{\pm 1}$.  

\begin{theorem}\label{T:BKNproj2}  Let ${\mathfrak g}$ be a Type I classical Lie superalgebra and 
$M\in {\mathcal F}$. Let $\{x_{i} \mid i\in I\}$ 
be a set of orbit representatives for the minimal orbits\footnote{By minimal orbit, we mean minimal non-zero orbit with respect 
to the partial order on orbits given by containment in closures.} of the action of $G_{\0}$ on $\fg_{1}$ and $\{y_{j}\mid j\in J\}$ 
be a set of orbit representatives for the minimal orbits of the action of $G_{\0}$ on $\fg_{-1}$.
\begin{itemize} 
\item[(a)] Then $M$ is projective in ${\mathcal F}$ if and only if $M$ is projective on restriction to $U(\langle x_{i} \rangle)$ 
for all $i\in I$ and to $U(\langle y_{j} \rangle)$ for all $j\in J$.
\item[(b)] Furthermore, assume that $\fg$ admits a strong duality $\tau$ and $M \cong M^{\tau}$. 
Then $M$ is a projective in ${\mathcal F}$ if and only if $M$ is projective on restriction to $U(\langle x_{i} \rangle)$ 
for all $i\in I$.
\end{itemize} 
\end{theorem}

\begin{proof}  (a) The module $M$ is not projective if and only if $\V_{\fg_{1}}(M)$ or 
$\V_{\fg_{-1}}(M)$ is nontrivial.  
However, $\V_{\fg_{\pm 1}}(M)$ is a closed $G_{\0}$-stable variety.  Thus 
$\V_{\fg_{\pm 1}}(M)$ is nonzero if and only if it contains a minimal $G_{\0}$-orbit.  
Therefore, $M$ is not projective if and only if $\V_{\fg_{\1}}(M)$ 
contains an $x_{i}$ ($i\in I$) or $\V_{\fg_{-1}}(M)$ contains a $y_{j}$ ($j\in J$).  
The result then follows by the rank variety description of $\V_{\fg_{\pm 1}}(M)$. 
Part (b) follows from (a) because of the isomorphism $\V_{\fg_{-1}}(M)\cong \V_{\fg_{1}}(M)$  when 
$M\cong M^{\tau}$. 
\end{proof}

\subsection{Examples} \label{SS:examples} We now examine the $G_{\0}$-orbit structure of $\fg_{1}$ and $\fg_{-1}$.  Remarkably, the following examples show that for many classical Lie superalgebras of Type I it suffices to test for Kac filtrations, dual Kac filtrations, and projectivity of a $\fg$-module using a \emph{single} element from $\fg_{1}$ and/or $\fg_{-1}$.  Note that each of the following examples satisfies the conditions of Theorem~\ref{T:BKNproj2} and, except for $\widetilde{\fp}(n)$ and $\fp (n)$, admits a strong duality.  In each case we use the matrix description of $\fg$ given in \cite[Section 2]{Kac1}.

\subsubsection{$\fg=\gl (m|n)$.}\label{SSS:glmn} In this case, $\fg_{1}$ and $\fg_{-1}$ each have a single nontrivial minimal $G_{\0}$-orbit. Namely, 
the action of $G_{\0}\cong GL(m)\times GL(n)$ on $\fg_{1}$ is given by $(A,B)\cdot x = AxB^{-1}$ for $A\in GL(m),\ B\in GL(n),
\ x\in \fg_{1}$.  The orbits are 
\[
(\fg_{1})_{r} = \{\, x\in\fg_{1}\mid \rk(x)=r\,\}
\]
for 
$0\le r\le \min(m,n)$. The closure of $(\fg_{1})_{r}$ is
\[
\overline{(\fg_{1})_{r}} = \{\, x\in\fg_{1}\mid \rk(x)\le r\,\};
\]
thus $(\fg_{1})_{r} \subset \overline{(\fg_{1})_{s}}$ if and only if $r \le s$.  Hence, the Hasse diagram\footnote{The graph which depicts the partial ordering given by inclusion of orbit closures.} is a simple chain. The closures are often 
referred to as \emph{determinantal varieties} (see \cite[Chaps. 9, 12]{Har:92} for details on determinantal varieties). 
The situation for $\fg_{-1}$ is analogous. Thus to test for Kac filtrations, dual Kac filtrations, and projectivity, it suffices to consider a single element from each of $\fg_{1}$ and $\fg_{-1}$.  Furthermore, recall from Section \ref{SS:duality} that $\fg$ admits a strong duality. Thus, 
Theorem~\ref{T:BKNproj2} reduces the test for projectivity to {\em a single rank one element of $\fg_{1}$} whenever $M \cong M^{\tau}$ (e.g.\ if $M$ is simple).

We should mention some notable facts about the $\fg_{\pm 1}$-support varieties in this case. Because the Hasse diagram is a chain, it follows that if $M\in {\mathcal F}$ then 
$\V_{{\mathfrak g}_{1}}(M)=\overline{(\fg_{1})_{r}}$ for some $r$ and $\V_{{\mathfrak g}_{-1}}(M)=\overline{(\fg_{-1})_{s}}$ for some $s$. 
So $\V_{{\mathfrak g}_{\pm 1}}(M)$ is always irreducible and can be computed by applying the rank variety description to a representative from each of the $\min(m,n)$ orbits. 

Let $L(\lambda)$ be the simple $\mathfrak{gl}(m|n)$-module corresponding to the highest weight $\lambda$. The atypicality 
of $\lambda$ (cf.\ \cite[Section 7.2]{BKN1}) is an important  combinatorial notion defined by using $\lambda$ and the roots of ${\mathfrak g}$. 
By using the description of ${\mathcal X}_{L(\lambda)}$ given in 
\cite[Theorem 5.4]{dufloserganova}, together with \eqref{relationship} and \eqref{relationship2}, we can conclude that 
\begin{equation} 
\V_{{\mathfrak g}_{\pm 1}}(L(\lambda))=(\fg_{\pm 1})_{r}
\end{equation} 
where $r=\text{atyp}(\lambda)$. Conversely, it is not difficult to see that if one has independently computed $\V_{\fg_{1}}(S)$ for all simple $\fg$-modules $S$, then using the method of odd reflections allows one to compute the Duflo-Serganova associated variety of any simple module.

\subsubsection{$\fg = \mathfrak{sl}(m|n)$ when $m\neq n$}\label{SSS:TypeA}  If $m\neq n$, then $\fg = \mathfrak{sl}(m|n) \subseteq \gl (m|n)$ consists of the matrices of supertrace zero.  This is the simple Lie superalgebra of type A.  Then we have 
\[
G_{\0} = \left\{(A,B) \in GL(m) \times GL(n) \mid \det (A) \det (B)^{-1}=1 \right\},
\]
and $\fg_{\pm 1}$ and the action of $G_{\0}$ on $\fg_{\pm 1}$ are as in the previous example.  It is not difficult to verify the $G_{\0}$-orbits are the same as the $GL(m) \times GL(n)$-orbits.  That is, if $(A,B) \cdot x =y$ for $x,y \in \fg_{1}$ and $(A,B) \in GL(m) \times GL(n)$, then one can choose $(A',B') \in G_{\0}$ such that  $(A',B') \cdot x =y$.  A similar statement applies to the action of $G_{\0}$ on $\fg_{-1}$.   Again the orbits of $G_{\0}$ on $\fg_{\pm 1}$ form a chain and there is a strong duality, as described in Section \ref{SS:duality}, and thus our remarks for $\gl (m|n)$ on detecting Kac filtrations, dual Kac filtrations, and projectivity apply here as well.


\subsubsection{$\fg =\mathfrak{sl}(n|n)$ and $\fg = \mathfrak{psl}(n|n)$}\label{SSS:TypeA2}  In the case when $m=n$ the Lie superalgebra $\mathfrak{sl}(n|n)$ has a one dimensional center given by scalar multiples of the identity matrix.  The Type A simple Lie superalgebra is then $\mathfrak{psl}(n|n)$, the quotient of $\mathfrak{sl}(n|n)$ by this center.  We now discuss how these situations are different from $\gl (n|n)$.  In both cases $\fg_{\pm 1}$ is as for $\gl (n|n)$.  

For $\mathfrak{sl}(n|n)$ one has that 
\[
G_{\0} \cong \left\{(A,B) \in GL(n) \times GL(n) \mid \det(A)\det(B)^{-1}=1 \right\}.
\]  For elements of $\fg_{1}$ with rank strictly less than $n$, then just as in the case when $m \neq n$ one can check that 
the $G_{\0}$-orbits coincide with the $GL(n) \times GL(n)$-orbits and once again the orbits are the 
$(\fg_{1})_{r}$ ($r=0, \dotsc , n-1$) with the Hasse diagram of these orbit closures forming a chain.  On the other hand, if $x \in \fg_{1}$ is a matrix of full rank, then using the results for $\gl (m|n)$ we easily see that the orbit containing $x$ contains a unique matrix which is a scalar multiple of the identity (the scalar being the determinant of $x$).  The orbits of full rank matrices form a one parameter family with each orbit determined by the determinant of a representative.  That is, the orbits of full rank matrices are the closed subvarieties $\det^{-1}(a)$ for $a \in \C^{\times}$.  To use Theorem~\ref{T:BKNproj2} one must choose one representative from the orbit of rank one matrices, and one from each orbit in the one parameter family of orbits of full rank matrices.  

However, using the structure of $\V_{\fg _{1}}(M)$ we can make further reductions.  Since it is a conical variety a full rank matrix will appear in $\V_{\fg_{1}}(M)$ if and only if every scalar multiple appears.  Consequently, $\V_{\fg_{1}}(M)$ contains a full rank matrix if and only if it contains all full rank matricies.  But the set of full rank matrices is dense in the set of all matrices in the Zariski topology.  Since $\V_{\fg_{1}}(M)$ is closed it follows that a full rank matrix appears in the variety only if a rank one matrix appears.  Thus $\V_{\fg_{1}}(M) \neq 0$ if and only if a representative of the orbit $(\fg_{1})_{1}$ appears.  A similar analysis applies to $\fg_{-1}$.  Combined with the strong duality, our remarks for $\gl (m|n)$ on detecting Kac filtrations, dual Kac filtrations, and projectivity apply here as well.

Finally, consider $\mathfrak{psl}(n|n)$.  Let $I = (I_{n}, I_{n}) \in GL(n) \times GL(n)$ where $I_{n}$ denotes the identity matrix.   In this case one has
\[
G_{\0} \cong \left\{(A,B) \in GL(n) \times GL(n) \mid \det(A)\det(B)^{-1}=1 \right\}/\C^{\times}I_{n}.
\]  Clearly the subgroup $\C^{\times}I_{n}$ fixes $\fg_{1}$ pointwise.  Hence, 
the orbit structure is precisely as it is in the $\mathfrak{sl}(n|n)$ case.  In both cases $\fg_{-1}$ is entirely analogous.

\subsubsection{$\fg = \mathfrak{osp}(2|2n)$ }\label{SSS:TypeC}  Let $\fg$ be the simple Lie superalgebra of type $C(n+1)$.  Then 
\[
G_{\0} \cong \C^{\times} \times Sp(2n)
\] and $\g_{1}\cong V_{2n}$, the natural module for $Sp(2n)$ with the action of $G_{\0}$ on $v \in V_{2n}$ given by $(a,B)\cdot v = aBv$ for all $(a,B) \in \C^{\times} \times Sp(2n)$.  The action of $Sp(2n)$ is transitive on $V_{2n}\smallsetminus\{0\}$. This is probably well known, and follows easily by showing that any nonzero vector in $V_{2n}$ is part of a basis having the same matrix of the skew form as the standard basis, using nondegeneracy of the form and induction on $n$. Thus there are only two orbits of $G_{\0}$ on $\fg_{1}$, namely $\{0\}$ and $\fg_{1}\smallsetminus\{0\}$. 

The case for $\fg_{-1}$ is argued similarly. Since the type $C$ Lie superalgebra is contragredient (cf.\ \cite[Section 2.5]{Kac1}), it has a strong duality.  Our remarks for $\gl (m|n)$ on detecting Kac filtrations, dual Kac filtrations, and projectivity apply here as well.

\subsubsection{$\fg = \tilde{\mathfrak{p}}(n)$ and $\fg =\mathfrak{p}(n)$}\label{SSS:TypeP}  Let 
$\fg = \tilde{\mathfrak{p}}(n)$ be the Lie superalgebra of all matrices in $\gl (n|n)$ which are invariant 
under a fixed odd nondegenerate symmetric form as in \cite{serganova}.  Then $G_{\0} \cong GL(n)$ and $g_{-1} 
\cong \Lambda^{2}(V^{*})$ and $g_{1} \cong S^{2}(V)$ as $G_{\0}$-modules, where $V$ denotes the natural $GL(n)$-module.  
In both cases we are in the situation of \cite[Examples 2.13, 2.14]{tevelev}.  Therefore, there are a finite number of 
orbits given again via a rank condition (in the matrix realization) and the Hasse diagram of 
their closure relation forms a chain.  In particular, similar to the situation for $\gl (m|n)$, one only needs one element from $\fg_{-1}$ and one from $\fg_{1}$ to test for Kac filtrations, dual Kac filtrations, and projectivity. However, there is no strong duality on $\tilde\fp(n)$ so we cannot reduce to a single element.

Now let $\fg = \mathfrak{p}(n) = [\tilde{\mathfrak{p}}(n),\tilde{\mathfrak{p}}(n)]$ be the simple Lie superalgebra of type $P(n-1)$.  Then $\fg_{-1}$ and $\fg_{1}$ are as above but $G_{\0} \cong SL(n)$.  This case is much like $\mathfrak{sl}(n|n)$.  Namely, the $GL(n)$-orbits corresponding to matrices of rank less than $n$ are also $SL(n)$-orbits.  However, among the matrices of rank $n$ one again obtains a one parameter family of orbits.  Each is a closed variety given by a fiber of the determinant map.  As with $\mathfrak{sl}(n|n)$, to apply Theorem~\ref{T:BKNproj2} one must choose a representative from each of these orbits as well as one from the orbit of minimal nonzero rank. However, we can again use the properties of the support varieties to further reduce to testing for Kac filtrations, dual Kac filtrations, and projectivity on a single element of $\fg_{1}$ and a single element of $\fg_{-1}$.

\subsection{Cartan Matrices} \label{SS:cartan}

The \emph{Cartan matrix} of $\mathcal{F}$ is an array which contains the 
multiplicities of the simple modules in the indecomposable projective modules. If $S$ and $T$ are two 
simple modules in ${\mathcal F}$, let $P(S)$ be the indecomposable projective cover of the simple module $S$, 
and $[P(S):T]$ denote the number of times the simple module 
$T$ appears as a composition factor of $P(S)$.

In the second statement of the following result we assume $\dim_{\C}\Hom_{\mathcal{F}}(S,S)=1$ for all simple modules in $\mathcal{F}$.  Note that when $\fg$ is a simple Lie superalgebra of Type I or $\gl (m|n)$, then $S$ is generated by a one dimensional highest weight space as a consequence of \cite[Proposition 2.5.5(a)]{Kac1}.  From this it immediately follows that 
$\Hom_{\F}(S,S)$ is one dimensional.  In general, the Cartan matrix of a Type I Lie superalgebra with strong duality is symmetric after multiplying on the left by the diagonal matrix with $\kappa_{S}^{-1}$ (cf.\ \eqref{E:Homs}) as the $S$th diagonal entry.

\begin{theorem} \label{T:symmetriccartan}
Assume $\fg$ is a Type I superalgebra with a strong duality.  Then the projective cover and injective envelope of each simple module are isomorphic.  

Furthermore, if 
\[
\dim_{\C}\Hom_{\mathcal{F}}(S,S)=1
\]
for all simple modules $S$ in $\mathcal{F}$, then the Cartan matrix of $\fg$ is symmetric.
\end{theorem}

\begin{proof} First note that if $P$ is an indecomposable projective module, then by Theorem~\ref{T:projective}(a)  
and the fact that $\tau \left(\V_{\fg_{\pm 1}}(P^{\tau}) \right)=\V_{\fg_{\mp 1}}(P)$ it follows that $P^{\tau}$ is 
also projective.  Furthermore $P^{\tau}$ has the same character as $P$ since the simple modules are fixed by $\tau$.  
Arguing  just as for Category $\mathcal{O}$ for semisimple Lie algebras (cf. \cite[Corollary 3.10]{Hum:08}) using 
\cite[Proposition 3.4.1(f)]{BKN2} one can show that a projective module is determined by its character.  Thus $P^{\tau} 
\cong P $.  In particular, $P(S)^{\tau}\cong P(S)$. But on the other hand 
$P(S)^{\tau}$ is an indecomposable injective module with socle $S^{\tau}\cong S$. Thus $P(S)^{\tau}\cong I(S)$, 
and hence $P(S)\cong I(S)$.

If $S$ and $T$ are two simple modules, then the multiplicity of $T$ as a composition factor of $P(S)$ is given by 
\begin{eqnarray*} 
[P(S):T]\, \kappa_{T}&=& \dim_{\C} \Hom_{\F}(P(T),P(S))\\
&=& \dim_{\C} \Hom_{\F}(P(S)^{\tau},P(T)^{\tau})\\
&=& \dim_{\C} \Hom_{\F}(P(S),P(T))\\
&=& [P(T):S]\, \kappa_{S}.
\end{eqnarray*} 
In particular, as we assume $\kappa_{S}=\kappa_{T}=1$, the Cartan matrix is symmetric.
\end{proof}

\section{Further Questions} 

The results in this paper provide foundational material for the study of 
the homological properties of the category ${\mathcal F}$ for classical Lie superalgebras. 
We list below several questions that arose from our findings which are relevant for the 
further development of this theory. 
\vskip .5cm 
\noindent
(4.1) According to Section~\ref{SS:selfinjective}, if $P(S)$ is a projective cover for a simple module $S$ in ${\mathcal F}$ then 
$P(S)\cong I(T)$ where $I(T)$ is the injective hull of a simple module $T$ in ${\mathcal F}$. Can 
one find a general formula (like the one for a finite-dimensional Hopf algebra, cf.\ \cite[I. 8.13 Proposition]
{jantzen}) which describes the relationship between $S$ and $T$?  
\vskip .25cm
\noindent
(4.2) Given a simple module $S$ in ${\mathcal F}$, can one give a concrete realization of 
$\End _{\mathcal F}(P(S))$? An answer to this question may contribute to a solution of (4.1).  

Note that recently Brundan and Stroppel considered a related question for $\gl (m|n)$  \cite[Introduction]{BrSt:08a}.  Namely, they show that if $A$ is the endomorphism algebra of a projective generator of a block in category $\mathcal{O}$, then $A$ is isomorphic to a limiting version of a diagrammatically defined algebra which, in turn, is a quasi-heredity cover of an algebra first introduced by Khovanov in his categorification of the Jones polynomial.  One consequence of their work is a proof that the algebra $A$ is Koszul. See also \cite{ChLa:09} for another approach.
\vskip .25cm
\noindent 
(4.3) In Section~\ref{SS:complexityrelcoho} we provide an interpretation of the complexity of 
a module in ${\mathcal F}$ using the relative cohomology. One fundamental and elusive question 
is whether one can construct a ``support variety'' 
for a given module $M$ in ${\mathcal F}$ whose dimension is equal 
to the complexity of $M$. 
Recall that when $M$ is contained in a block with only finitely many simple modules then 
$c_{\mathcal F}(M)=\dim V_{({\mathfrak g},{\mathfrak g}_{\0})}(M)$. 
\vskip .5cm 
\noindent

\let\section=\oldsection
\bibliographystyle{amsmath}
\bibliography{BKN3}
\def\Dbar{\leavevmode\lower.6ex\hbox to 0pt{\hskip-.23ex \accent"16\hss}D}
  \def\cprime{$'$} \def\germ{\mathfrak}\def\scr{\mathcal}

\end{document}